\documentclass[12pt,a4paper]{article}
\usepackage[active]{srcltx}
\usepackage{amsfonts}
\usepackage{amsmath}
\usepackage{amsbsy}
\usepackage{amsxtra}
\usepackage{latexsym}
\usepackage{amssymb}
\usepackage{url}
\usepackage{pstricks,pst-node}
\makeatletter
\g@addto@macro\thesection.
\makeatother

\newtheorem{theorem}{Theorem}
\newtheorem{lemma}{Lemma}
\newtheorem{corollary}{Corollary}

\newtheorem{remark}{Remark}

\DeclareMathOperator{\cone}{cone}
\DeclareMathOperator{\inte}{int}
\DeclareMathOperator{\aff}{aff}
\DeclareMathOperator{\icore}{icore}
\DeclareMathOperator{\bdr}{bdr}

\DeclareMathOperator{\co}{co}

\newcommand{\lng}{\langle}
\newcommand{\rng}{\rangle}
\newcommand{\lf}{\left}
\newcommand{\rg}{\right}
\newcommand{\sa}{\sqcap}
\newcommand{\su}{\sqcup}
\newcommand{\R}{\mathbb R}
\newcommand{\N}{\mathbb N}

\newcommand{\f}{\frac}

\newenvironment{proof}{{\noindent\bf Proof.}}{\hfill$\Box$\\}

\begin{document}

\title{Self-dual cones, generalized lattice operations and isotone
projections
\thanks{{\it 1991 A M S Subject Classification:} Primary 90C33,
Secondary 15A48; {\it Key words and phrases:} convex sublattices, isotone projections. }}
\author{A. B. N\'emeth\\Faculty of Mathematics and Computer Science\\Babe\c s Bolyai University, Str. Kog\u alniceanu nr. 1-3\\RO-400084 Cluj-Napoca, Romania\\email: nemab@math.ubbcluj.ro \and S. Z. N\'emeth\\School of Mathematics, The University of Birmingham\\The Watson Building, Edgbaston\\Birmingham B15 2TT, United Kingdom\\email: nemeths@for.mat.bham.ac.uk}
\date{}
\maketitle

\begin{abstract}

By using the metric projection onto a closed self-dual cone of the Euclidean space,
M. S. Gowda, R. Sznajder and J. Tao have defined generalized lattice operations,
which in the particular case of the nonnegative orthant of a Cartesian reference
system reduce to the lattice operations of the coordinate-wise ordering. 
The aim of the present note is twofold: to give a geometric characterization
of the closed convex sets which are invariant with respect to these
operations, and to relate this invariance property to the isotonicity of the 
metric projection onto these sets. As concrete examples the Lorentz cone and the 
nonnegative orthant are considered. Old and recent results on closed convex
Euclidean sublattices  due to D. M. Topkis, A. F. Veinott and to M. Queyranne and  F. Tardella
, respectively are obtained as particular cases. The topic is related to 
variational inequalities where the isotonicity of the metric projection is 
an important technical tool. For Euclidean sublattices this approach was considered 
by G. Isac, H. Nishimura and E. A. Ok.

\end{abstract}

\section{Introduction}

A commonly used approach in establishing the solvability of variational inequalities and 
furnishing their solution is the usage of fixed point theorems and the iterative processes they
engender, respectively (e.g., \cite{Auslender1976,Bertsekas1989,Iusem1997,Khobotov1987,Korpelevich1976,Marcotte1991,Nagurney1993,Sibony1970,Solodov1999,Solodov1996,Sun1996}).

A specific route to follow during this endeavour is to derive monotone and convergent iterative
processes with respect to some order relations. For the particular case of nonlinear 
complementarity problems this approach was first initiated by G. Isac and A. B. N\'emeth. Both 
the solvability and the approximation of solutions of nonlinear complementarity problems 
can be handled by using the metric projection onto the convex cone associated with the problem. 
The idea to relate the ordering induced by the convex cone and the metric projection onto the 
convex cone goes back to their paper \cite{IsacNemeth1986}, where a convex cone in the 
Euclidean space which admits an isotone projection onto it (called 
\emph{isotone projection cone}) was characterized. The isotonicity is considered with respect 
to the order induced by the convex cone.
%

The isotone projection cones were used in the solution of
some nonlinear complementarity problems \cite{IsacNemeth1990b}, \cite{IsacNemeth2008c},
\cite{Nemeth2009a}. Solving complementarity problems by successive approximation
require repeated projection onto the underlying cone. It is particularly meaningful that this 
is an efficient procedure for isotone projection cones \cite{NemethNemeth2009}.


If the projection onto the closed convex set encountered in the definition
of a variational inequality is monotone with respect to an appropriate order
relation, then an iterative method can be worked out for its solution.
An easily handleable order relation in the Euclidean space is the coordinate-wise 
ordering. G. Isac \cite{Isac1996} showed that the projection onto a closed convex set is 
isotone with respect to this order relation if the set is a sublattice.

In a recent paper H. Nishimura and E. A. Ok \cite{NishimuraOk2012}
showed that latticiality is also a necessary condition for the isotonicity
of the metric projection. In the last cited paper several applications
were given for variational inequalities defined on closed convex sublattices and other related
equilibrium problems.
But how do the closed convex sublattices with nonempty interior of the coordinate-wise
ordered Euclidean space look? The answer to this question seems to go back to
the results of D. M. Topkis \cite{Topkis1976} and A. F. Veinott Jr. \cite{Veinott1981}
and was settled recently by M. Queyranne and  F. Tardella \cite{QueyranneTardella2006}.


The positive cone of the coordinate-wise ordering is the nonnegative orthant
of a Cartesian reference system in the Euclidean space. It is a self-dual latticial cone and 
defines well behaved lattice operations. Although largely investigated, they are very 
restrictive. Among the attempts to extend these lattice operations, one concerning self-dual
cones and intrinsically related to metric projections is that proposed
by M. S. Gowda, R. Sznajder and J. Tao \cite{GowdaSznajderTao2004}.
Fortunatelly these extended lattice operations, apart from keeping several properties of 
lattice operations, seem to be good tools in handling the problem of the isotonicity of 
the metric projections. 

In this note we characterize the closed convex sets which are invariant with respect to these 
operations showing that the metric projection onto these sets is isotone with respect to the 
order generated by the self-dual cone giving rise to the respective operations.

The structure of this paper is as follows. In Section \ref{Self-dual cones} we will define
the notion of self-dual cones and as particular examples the nonnegative orthant and the 
Lorentz cone. In Section \ref{Generalized lattice operations} we will define the lattice
operations for the nonnegative orthant and extend these operations to a self-dual cone. In
the same section we state our main results, namely Theorems \ref{FOO}, \ref{ISO} and 
\ref{POLYH}, which will be proved in Sections \ref{spt1}, \ref{spt2} and \ref{spt3}
In Section \ref{prop ext latt} we will give a series of properties for
the extended lattice operations defined by a self-dual cone used in the later sections. 
Sections \ref{sec inv} and \ref{sec iso} contain several lemmas needed to prove our main 
results: Theorems \ref{FOO}, \ref{ISO} and \ref{POLYH}. Exception is Lemma \ref{minimal} 
which together with Corollary \ref{ketdiminv} are used in Sections \ref{sect no} 
and \ref{sect lo} only. However the above lemma and corollary exhibit fundamental 
geometric properties of the extended lattice operations. The main results are also 
motivated by the particular cases of the nonnegative orthant and Lorentz cone investigated
in Sections \ref{sect no} and \ref{sect lo}, respectively. Finally, we end our paper by
making some comments and raising some open questions in Section \ref{concl} 

\section{Self-dual cones}\label{Self-dual cones}

Denote by $\R^m$ the $m$-dimensional Euclidean space endowed with the scalar 
product $\lng\cdot,\cdot\rng:\R^m\times\R^m\to\R,$ and the Euclidean norm $\|.\|$ and topology 
this scalar product defines.

Throughout this note we shall use some standard terms and results from convex geometry 
(see e.g. \cite{Rockafellar1970}). 

Let $K$ be a \emph{convex cone} in $\R^m$, i. e., a nonempty set with
(i) $K+K\subset K$ and (ii) $tK\subset K,\;\forall \;t\in \R_+ =[0,+\infty)$.
The convex cone $K$ is called \emph{pointed}, if $(-K)\cap K=\{0\}.$

The cone $K$ is {\it generating} if $K-K=\R^m$.
 
For any $x,y\in \R^m$, by the equivalence $x\leq_K y\Leftrightarrow y-x\in K$, the 
convex cone $K$ 
induces an {\it order relation} $\leq_K$ in $\R^m$, that is, a binary relation, which is 
reflexive and transitive. This order relation is {\it translation invariant} 
in the sense that $x\leq_K y$ implies $x+z\leq_K y+z$ for all $z\in \R^m$, and 
{\it scale invariant} in the sense that $x\leq_Ky$ implies $tx\leq_K ty$ for any $t\in \R_+$.
If $\leq$ is a translation invariant and scale invariant order relation on $\R^m$, then 
$\leq=\leq_K$ with $K=\{x\in\R^m:0\leq x\}.$ If $K$ is pointed, then $\leq_K$ is 
\emph{antisymmetric} too, that is $x\leq_K y$ and $y\leq_K x$ imply that $x=y.$
The elements $x$ and $y$ are called \emph{comparable} if $x\leq_K y$
or $y\leq_K x.$

We say that $\leq_K$ is a \emph{latticial order} if for each pair
of elements $x,y\in \R^m$ there exist the lowest upper bound
$\sup\{x,y\}$ and the uppest lower bound $\inf\{x,y\}$ of
the set $\{x,y\}$ with respect to the order relation $\leq_K$.
In this case $K$ is said a \emph{latticial or simplicial cone},
and $\R^m$ equipped with a latticial order is called an
\emph{Euclidean vector lattice}.

The \emph{dual} of the convex cone $K$ is the set
$$K^*:=\{y\in \R^n:\;\lng x,y\rng \geq 0,\;\forall \;x\in K\},$$
with $\lng\cdot,\cdot\rng $ the standard scalar product in $\R^n$.

The cone $K$ is called \emph{self-dual}, if $K=K^*.$ If $K$
is self-dual, then it is a generating, pointed, closed cone.

In all that follows we shall suppose that $\R^m$ is endowed with a
Cartesian reference system with the standard unit vectors $e_1,\dots,e_m$.
That is, $e_1,\dots,e_m$ is an orthonormal system of vectors in the sense
that $\lng e_i,e_j\rng =\delta_i^j$, where $\delta_i^j$ is the Kronecker symbol.
Then, $e_1,...,e_m$ form a basis of the vector space $\R^m$. If $x\in \R^m$, then
$$x=x^1e_1+...+x^me_m$$
can be characterized by the ordered $m$-tuple of real numbers $x^1,...,x^m$, called
\emph{the coordinates of} $x$ with respect the given reference system, and we shall write
$x=(x^1,...,x^m).$ With this notation we have $e_i=(0,...,0,1,0,...,0),$
with $1$ in the $i$-th position and $0$ elsewhere. Let
$x,y\in \R^m$, $x=(x^1,...,x^m)$, $y=(y^1,...,y^m)$, where $x^i$, $y^i$ are the coordinates of
$x$ and $y$, respectively with respect to the reference system. Then, the scalar product of $x$
and $y$ is the sum
$\lng x,y\rng =\sum_{i=1}^m x^iy^i.$

The set
\[\R^m_+=\{x=(x^1,...,x^m)\in \R^m:\; x^i\geq 0,\;i=1,...,m\}\]
is called the \emph{nonnegative orthant} of the above introduced Cartesian
reference system. A direct verification shows that $\R^m_+$ is a
self-dual cone.

The set
\[L_{m+1}=\{(x,x^{m+1})\in\R^{m}\otimes\R=\R^{m+1}:\;\|x\|\leq x^{m+1}\},\] (or simply $L$ if 
there is no confusion about the dimension) is a self-dual cone called 
\emph{$m+1$-dimensional second order cone}, or 
\emph{$m+1$-dimensional Lorentz cone}, or \emph{$m+1$-dimensional ice-cream cone} 
(\cite{GowdaSznajderTao2004}). 

The nonnegative orthant $\R^m_+$ and the Lorentz cone $L$ defined above are the
most important and largery used self-dual cones in the Euclidean space.
But the family of self-dual cones is rather rich
\cite{BarkerForan1976}.

\section{Generalized lattice operations}\label{Generalized lattice operations}

A \emph{hypersubspace} or a \emph{hyperplane through the origin}, is a set of form
\begin{equation}\label{hypersubspace}
H(u,0)=\{x\in \R^m:\;\lng u,x\rng =0\},\;\;u\not=0.
\end{equation}
For simplicity the hypersubspaces will also be denoted by $H$.
The nonzero vector $u$ in the above formula is called \emph{the normal}
of the hyperplane. 

A \emph{hyperplane} (through $a\in\R^m$) is a set of form
\begin{equation}\label{hyperplane}
H(u,a)=\{x\in \R^m:\;\lng u,x\rng =\lng u,a\rng, \;u\not= 0\}.
\end{equation}

A hyperplane $H(u,a)$ determines two \emph{closed halfspaces} $H_-(a,u)$ and
$H_+(u,a)$  of $\R^m$, defined by

\[H_-(u,a)=\{x\in \R^m:\; \lng u,x\rng \leq \lng u,a\rng\},\]
and
\[H_+(u,a)=\{x\in \R^m:\; \lng u,x\rng \geq \lng u,a\rng\}.\]

Taking a Cartesian reference system in $\R^m$ and using the above introduced
notations,
the 
\emph{coordinate-wise order}  $\leq$ in $\R^m$ is defined by
\[x=(x^1,...,x^m)\leq y=(y^1,...,y^m)\;\Leftrightarrow\;x^i\leq y^i,\;i=1,...,m.\]
Using the notion of the order relation induced by a cone, defined in the preceding
section, we see that $\leq =\leq_{\R^m_+}$.

With the above representation of $x$ and $y$, we define
$$x\wedge y=(\min \{x^1,y^1\},...,\min \{x^m,y^m\}),\;\;\textrm{and}\;\;x\vee y=(\max \{x^1,y^1\},...,\max \{x^m,y^m\}).$$

Then, $x\wedge y$ is  the uppest lower bound and $x\vee y$ is the lowest upper bound of 
the set $\{x,y\}$ with respect to the coordinate-wise order. Thus, $\leq$ is a lattice order in $\R^m.$
The operations $\wedge$ and $\vee$ are called \emph{lattice operations}.

The subset $M\subset \R^m$ is called a \emph{sublattice of
the coordinate-wise ordered Euclidean space} $\R^m$, if from
$x,y\in M$ it follows that $x\wedge y,\;x\vee y\in M.$ 

Denote by $P_D$ 
the projection mapping onto a nonempty closed convex set $D\subset \R^m,$ 
that is the mapping which associate
to $x\in \R^m$ the unique nearest point of $x$ in $D$ (\cite{Zarantonello1971}):

\[ P_Dx\in D,\;\; \textrm{and}\;\; \|x-P_Dx\|= \inf \{\|x-y\|: \;y\in D\}. \]

The nearest point $P_Dx$ can be characterized by

\begin{equation}\label{charac}
P_Dx\in D,\;\;\textrm{and}\;\;\lng P_Dx -x,P_Dx-y\rng \leq 0 ,\;\forall y\in D.
\end{equation}

From the definition of the projection and the characterization (\ref{charac}) there follow immediately 
the relations: 

\begin{equation}\label{en}
	P_D(-x)=-P_{-D}x ,
\end{equation}

\begin{equation}\label{et}
	P_{x+D}y=x+P_D(y-x) 
\end{equation}
for any $x,y\in\R^m$,

\begin{equation}\label{sun}
P_D(tx+(1-t)P_Dx)=P_Dx,\;\;\forall \;t\in[0,1].
\end{equation}

\emph{In all what follows next $K\subset \R^m$ will denote a self-dual cone.}

Define the following operations in $\R^m$: 
\[x\sa y=P_{x-K}y,\;\,\textrm{and}\;\; x\su y=P_{x+K}y,\]
(\cite{GowdaSznajderTao2004}). Assume 
the operations $\su$ and $\sa$ have precedence over the addition of vectors and 
multiplication of vectors by scalars.

A direct checking yields that if $K=\R^m_+$, then $\sa =\wedge$, and $\su =\vee$.
That is $\sa$ and $\su$ are some \emph{generalized lattice operations}.
Moreover: $\sa$ and $\su$ \emph{are lattice operations if and only if
the self-dual cone used in their definitions is a nonnegative orthant of
some Cartesian reference system.}

The subset $M\subset \R^m$ is called \emph{invariant with respect to
$\sa$ and $\su$} if for any $x,y\in M$ we have $x\sa y,\;x\su y\in M.$
That is, such an invariant set is the analogous for generalized lattice operations
of a sublattice for lattice operations.

We are now ready to state our main results in form of three theorems,
namely Theorems \ref{FOO}, \ref{ISO} and \ref{POLYH}, which 
will be proved in Sections \ref{spt1}, \ref{spt2} and \ref{spt3}, respectively.

\begin{theorem}\label{FOO}
The closed convex set $C\subset\R^m$ with nonempty interior is invariant
with respect to the operations $\sa$ and $\su$ defined by some self-dual cone
if and only if it is of form
\begin{equation}\label{vegso}
C=\cap_{i\in \N} H_-(u_i,a_i),
\end{equation}
where each hyperplane $H(u_i,a_i)$ is tangent to $C$ and is invariant with respect to $\sa$ and $\su$.
\end{theorem}

The interest of this theorem resides in the reduction of the
problem of invariance of a closed convex set $C\subset \R^m$ with nonempty interior 
with respect to the operations $\sa$ and $\su$ to the
characterization of the hyperplanes with this property
in the representation (\ref{vegso}) of $C$. In the important
case of the Lorentz cone and respective the nonnegative orthant
the invariant hyperplanes have rather simple geometric characterizations.

As we have remarked, in the case of $K=\R^m_+$ the invariant sets are
the so called sublattices of the coordinate-wise ordered Euclidean space.
As far as we know, the geometric characterization of closed convex sublattices
of the coordinate-wise ordered Euclidean space goes back to  D. M. Topkis \cite{Topkis1976}
and A. F. Veinott \cite{Veinott1981} and it was revisited recently by
M. Queyranne and F. Tardella \cite{QueyranneTardella2006}.
The above theorem can be considered the generalization of the
main result in the last cited paper with the remark that there the
hyperplanes in (\ref{vegso}) are geometrically
characterized. (We shall give the characterization of these hyperplanes
using an independent proof in the final section of our note, giving this way
a different proof of the main result in \cite{QueyranneTardella2006}.)

Let $\prec$ a given order relation in $\R^m$.
A closed convex set $C$ is called \emph{isotone projection set} 
and $P_C$ \emph {isotone projection with respect to $\prec$} if $P_C$ is order 
preserving with respect to $\prec$, i.e., if $x\prec y$ implies that $P_Cx\prec P_Cy.$
In all what follows we take $\prec =\leq_K$ with a given fixed self-dual cone $K\subset\R^m.$
Since there is no ambiguity, we shall use $\leq$ in place of $\leq_K$ and the term 
\emph{isotone projection} in place of isotone projection with respect to $\leq$. 

\begin{theorem}\label{ISO}
Let $K\subset \R^m$ be a self-dual cone and $\sa$ and $\su$ the above
generalized lattice operations defined with the aid of $K$.
Let $C\subset\R^m$ be a nonempty closed convex set. If $C$ is invariant with 
respect to the operations $\sa$ and $\su$, then $C$ is an isotone projection set.
\end{theorem}

This result for $K=\R^m_+$ is due to G. Isac \cite{Isac1996}. As have remarked recently
H. Nishimura and E. A. Ok \cite{NishimuraOk2012}, for this case the converse of theorem is also true:
from the isotonicity of $P_C$ it follows that $C$ is a sublattice.

Let $M\subset \R^m$ be a nonempty, closed convex set.
The nonempty subset $M_0\subset M$ is a \emph{face} of
$M$, if from $x,y\in M$ and $tx+(1-t)y \in M_0$, for some
$t\in ]0,1[$, it follows that $x,y\in M_0.$ The face $M_0\subset M$
is called \emph{proper face of $M$} if $M_0\not= M.$

If $\inte M \not= \emptyset$ and $M_0$ is a face
of $M$ with $\dim M_0 =m-1,$ then $M_0$ is
called a \emph{hyperface of $M$}.

The subset
\begin{equation}\label{polyhedron}
C=\cap_{i=1}^q H_-(u_i,a_i),
\end{equation}
is called a \emph{polyhedron}.

Suppose that $\inte C\not= \emptyset$ and that
the representation (\ref{polyhedron}) is \emph{sharp} in the sense that
no member
in the intersection representing $C$ is redundant.
Then, 
$$C^i=C\cap H(u_i,a_i)$$
is a hyperface of $C$ $(i=1,...,q)$, and the
normal $u_i$ in the representation of $H_-$ will be
called a \emph{normal of} $C$  $(i=1,...,q).$ Obviously, $\aff C^i=H(u_i,a_i),$
where $\aff C^i$ denotes the affine hull of $C^i$.
In the particular case of a polyhedron $C$ with nonempty interior, we can
strengthen and join the results in Theorem \ref{FOO} and Theorem \ref{ISO} as follows:

\begin{theorem}\label{POLYH}
Let $C$ be a polyhedron with nonempty interior, represented by 
\begin{equation}\label{redu}
C=\cap_{i=1}^q H_-(u_i,a_i),
\end{equation}
where the representation (\ref{redu}) is sharp in the
sense that each set $H(u_i,a_i)\cap C$ is a hyperface of $C$.
Suppose further that $K$ is a self-dual cone and $\sa$ and $\su$
are the generalized lattice operations defined with the aid of it.

Then, the following assertions are equivalent:
\begin{enumerate}
\item [\emph{(i)}] The polyhedron $C$ is a invariant set with respect to the
operations $\sa$ and $\su$;
\item [\emph{(ii)}] The projection $P_C$ is isotone with respect to the order
relation defined by $K$;
\item [\emph{(iii)}] Each hyperplane $H(u_i,a_i),\;i=1,...,q$ is invariant with respect to
the operations $\sa$ and $\su$;
\item [\emph{(iv)}] Each hyperplane $H(u_i,a_i),\;i=1,...q$ is an isotone projection set;
\item [\emph{(v)}] Each proper face of $C$ is invariant with respect to the
operations $\sa$ and $\su$.
\end{enumerate}
\end{theorem}

\section{Properties of $\sa$ and $\su$}\label{prop ext latt}

In the particular case of the self-dual cone $K\subset \R^m$, J. Moreau's theorem (\cite{Moreau1962}) reduces to
the following lemma:

\begin{lemma}\label{lm}
	For any $x$ in $K$ we have $x=P_Kx-P_K(-x)$ and $\lng P_Kx,P_K(-x)\rng=0$. The
	relation $P_Kx=0$ holds if and only if $x\in -K$.
\end{lemma}

\begin{lemma}\label{ll}
	The following relations hold for any $x,y,z,w\in\R^m$ and any real scalar 
	$\lambda>0$. 
	\begin{enumerate}
		\item[\emph{(i)}] $x\sa y=x-P_K(x-y)=y-P_K(y-x)$ and 
			$x\su y=x+P_K(y-x)=y+P_K(x-y)$.
		\item[\emph{(ii)}] $x\sa y=y\sa x$  and $x\su y=y\su x$.
		\item[\emph{(iii)}] $x\sa y\le x$ and $x\sa y\le y$, and equalities hold if and 
			only if $x\le y$ and $y\le x$, respectively.
		\item[\emph{(iv)}] $x\le x\su y$ and $y\le x\su y$, and equalities hold if and 
			only if $y\le x$ and $x\le y$, respectively.
		\item[\emph{(v)}] $x\sa y+x\su y=x+y$ 
		\item[\emph{(vi)}] $(x+z)\sa (y+z)=x\sa y+z$ and $(x+z)\sa (y+z)=x\sa y+z$.
		\item[\emph{(vii)}] $(\lambda x)\sa (\lambda y)=\lambda x\sa y$ and 
			$(\lambda x)\su (\lambda y)=\lambda x\su y$.
		\item[\emph{(viii)}] $\lng x-x\sa y,x\su y-x\rng=0,$
		\item[\emph{(ix)}] $(-x)\su(-y)=-x\sa y$.
		\item[\emph{(x)}] $\|x\su y-z\su w\|\leq \frac{2}{3}(\|x-z\|+\|y-w\|)$ and
		$\|x\sa y-z\sa w\|\leq \frac{2}{3}(\|x-z\|+\|y-w\|)$.
		\item[\emph{(xi)}]
		\[x\sa y=z\sa w,\;\;\forall \;\;z=\lambda x+(1-\lambda)x\sa y,\;\;w=\mu y+(1-\mu)x\sa y,\;\,\lambda,\;\mu \in [0,1],\]
	  \[x\su y=z\su w,\;\;\forall \;\;z=\lambda x+(1-\lambda)x\su y,\;\;w=\mu y+(1-\mu)x\su y,\;\,\lambda,\;\mu \in [0,1].\]
  		\item[\emph{(xii)}] If $x\sa y=0$, then $\lng x,y\rng$=0.
	\end{enumerate}
\end{lemma}

\begin{proof}
	\begin{enumerate}
		\item[(i)] From equation (\ref{et}) and Lemma \ref{lm} we have 
			\[
			\begin{array}{rcl}
				x\su y & = & P_{x+K}y=x+P_K(y-x)=
				x+(P_K(y-x)-P_K(x-y))+P_K(x-y)\\
				& = & x+(y-x)+P_K(x-y)=y+P_K(x-y).
			\end{array}
			\]
			A similar argument with $-K$ replacing $K$ shows that
			$x\sa y=x-P_K(x-y)=y-P_K(y-x)$.
		\item[(ii)] It follows easily from item (i).
		\item[(iii)] Since $x\sa y\in x-K$, it follows that $x\sa y\le x$. By using 
			item (ii) and the latter relation with $x$ and $y$ swapped, we get 
			$x\sa y=y\sa x\le y$. By item (i), the equality $x\sa y=x$ is 
			equivalent to $P_K(x-y)=0$. By Lemma \ref{lm}, the latter relations
			is equivalent to $x\le y$.
		\item[(iv)] It can be shown similarly to item (iii).
		\item[(viii)] By using item (i) and Lemma \ref{lm}, we get
			\[\lng x-x\sa y,x\su y-x\rng=\lng P_K(x-y),P_K(y-x)\rng=0.\]
	\end{enumerate}

	Items (v) and (vi) follow immediately from item (i). Item (vii) 
	follows easily from the positive homogeneity of $P_K$ and item (i). Item (ix) follows
	from \eqref{en} and item (i).
	
	To verify item (x) we use item (i) and the Lipschitz property of the
	metric projection (\cite{Zarantonello1971}), we obtain:
	\[\|x\su y-z\su w\| =\|x-P_K(x-y)-z+P_K(z-w)\|\leq \|x-z\|+\|P_K(x-y)-P_K(z-w)\|\leq \]
	\[\|x-z\|+\|(x-y)-(z-w)\|\leq 2\|x-z\|+\|y-w\|,\]
	and by symmetry
	\[\|x\su y-z\su w\| \leq \|x-z\|+2\|y-w\|.\]
	By adding the obtained two relations we conclude the first relation in item (x).
	The second relation can be deduced similarly. 
	
	Using the definition of $x\sa y$ we have according to the formula (\ref{sun}) that
	\[x\sa y=P_{x+K}y=P_{x+K}(\mu y+(1-\mu)x\sa y)=P_{x+K}w =x\sa w=P_{w+K}x.\]
	Using a similar argument we see that
	\[x\sa w=z\sa w.\]
	This is the first formula in item (xi). 
	A similar argument yields the second relation in this item.
	
	Item (xii) follows easily from items (v) and (viii).
	\end{proof}

\section{Subsets invariant with respect to $\sa$ and $\su$}\label{sec inv}

To shorten the writing the term \emph{invariant} from now on will mean
\emph{invariant with respect to the operations $\sa$ and $\su$ defined with the
aid of the given self-dual cone $K$.}

\begin{lemma}\label{minimal}
$\,$

\begin{enumerate}
\item[\emph{(i)}] The minimal invariant set containing the points $x,y\in \R^m$ is the set
$\{x,y\}$ if $x$ and $y$ are comparable, and the set $\{x,y,x\sa y,x\su y\}$
if $x$ and $y$ are not comparable;
\item[\emph{(ii)}] The minimal invariant convex set containing the points $x,y\in \R^m$ is the
closed line segment $[x,y]$ if $x$ and $y$ are comparable, and the 
planar rectangle with vertices $x$, $y$, $x\sa y$ and $x\su y$
if $x$ and $y$ are not comparable.
\end{enumerate}
\end{lemma}
\begin{proof}
The assertion (i) is the direct consequence of items (iii) and (iv) in Lemma \ref{ll}.

If $x$ and $y$ are comparable, then any two points in the segment $[x,y]$
are comparable and their set is invariant by (i). Hence, $[x,y]$
is invariant, and being the minimal convex set containing $x$ and $y$,
we arrive to the first assertion in item (ii).

If $x$ and $y$ are  not comparable, by items (v) and (viii) of Lemma \ref{ll},
$x$, $y$, $x\sa y$ and $x\su y$ form a spatial quadruple with all
the angles being rightangles. Hence, it must be a planar rectangle
denoted by $\Pi (x,y).$ The sides of this rectangle have comparable endpoints,
hence the whole boundary of $\Pi(x,y)$ must be contained in any invariant 
convex set containing $x$ and $y$. Let $v$ be an arbitrary point in $\Pi (x,y).$ 
The line trough $v$ parallel with the segment $[y,x\sa y]$ intersects the
segment $[x,x\sa y]$ in $z=\lambda x +(1-\lambda)x\sa y$, the line
through $v$ parallel with $[x,x\sa y]$ meets $[y,x\sa y]$ at
$w=\mu y+(1-\mu)x\sa y,$ with some $\lambda,\;\mu \in [0,1].$ See the below figure.
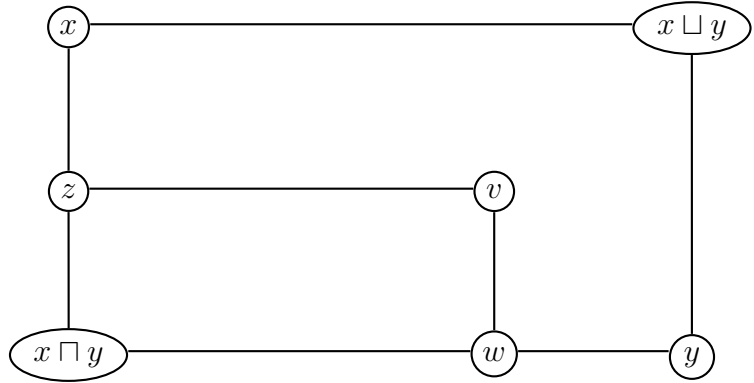
\begin{center}
\begin{psmatrix}
	\pscirclebox{$x$}      & & &     & \psovalbox{$x\su y$} 	\\
	\pscirclebox{$z$}	 & & & \pscirclebox{$v$} &		\\
	\psovalbox{$x\sa y$} & & & \pscirclebox{$w$} & \pscirclebox{$y$}
	\ncline{1,1}{1,5}
	\ncline{1,1}{2,1}
	\ncline{2,1}{3,1}
	\ncline{2,1}{2,4}
	\ncline{2,4}{3,4}
	\ncline{1,5}{3,5}
	\ncline{3,1}{3,4}
	\ncline{3,4}{3,5}
\end{psmatrix}
\end{center}
Obviously, the rectangle with vertices $z,\;x\sa y,\;w,\;v$ is contained in the rectangle $\Pi(x,y)$,
since they have the common points $z,\;x\sa y,\;w$. The same is true for
the rectangle with the vertices $z$, $x\sa y$, $w$ and $z\su w.$
Hence, the vertices $v$ and $z\su w$ must coincide, that is, 
$v=z\su w\in \Pi(x,y).$ 
Hence, every point in the considered
rectangle must be contained in any invariant convex set containing 
$x$ and $y$ and thus  
the whole rectangle $\Pi(x,y)$ is contained in any invariant convex set
containing the points $x$ and $y$.

We have to verify that $\Pi(x,y)$ itself is invariant. 
Take $u,v\in \Pi (x,y).$ If $u$ and $v$ are comparable, then
they form an independent set. If not, we argue as follows.
The lines through $u$ and $v$ parallel with the sides $[x,x\su y]$
and $[y,x\su y]$, respectively form a rectangle with opposite vertices $u$ and $v$. Denote
by $p$ and $q$ its other opposite vertices. See the below figure.
\begin{center}
\begin{psmatrix}
	 \pscirclebox{$x$}  & \pscirclebox{$z$}   & & &     & \psovalbox{$x\su y$} 	\\
	 &  \pscirclebox{$u$} &  & \pscirclebox{$q$} &    &		\\
		 & \pscirclebox{$p$} & & \pscirclebox{$v$} &    & \pscirclebox{$w$}	\\
		  \psovalbox{$x\sa y$} && & & 	   & \pscirclebox{$y$}
	\ncline{1,1}{4,1}
	\ncline{1,1}{1,2}
	\ncline{1,2}{1,6}
	\ncline{1,2}{2,2}
	\ncline{2,2}{2,4}
	\ncline{2,2}{3,2}
	\ncline{3,2}{3,4}
	\ncline{2,4}{3,4}
	\ncline{1,6}{3,6}
	\ncline{3,4}{3,6}
	\ncline{3,6}{4,6}
	\ncline{4,1}{4,6}
\end{psmatrix}
\end{center}
A reasoning as above, combined with a case analysis shows that
the rectangle with vertices $u,\;p,\;v,\;q$ must be contained in $\Pi (x,y)$
and hence it must coincide with $\Pi(u,v)$. 

Indeed, assume e. g. that
$u\in [p,z]$ with $z\in [x,x\su y]$, and  $v\in [p,w]$ with $w\in [y,x\su y]$.
Then, a reasoning as above shows that $p=z\sa w$, and using item (xi)
in Lemma \ref{ll} we see that $p=u\sa v.$ Thus, $u\sa v\in \Pi (x,y)$.

We can similarly see that $u\su v\in \Pi (x,y).$

\end{proof}

\begin{lemma}\label{latprop}
	If $M,\;M_i,\;\subset\R^m,\;\;i\in \mathcal I$ are invariant sets, then
	\begin{enumerate}
		\item[\emph{(i)}] $\cap_{i\in \mathcal I} M_i$ is also invariant,
		\item[\emph{(ii)}] $\eta M+a$ is also invariant for any $a\in\R^m$ and $\eta\in\R$.
		\item[\emph{(iii)}] If the nonempty convex set $C$ is invariant, then its affine hull denoted by
		$\aff C$ is invariant too.
		\item[\emph{(iv)}] The nonempty set $M\subset \R^m$ is an invariant convex set if and only if
		together with each pair $x,y$ of elements the convex hull 
		$\co\{x,y,x\sa y,x\su y\}$ is contained in $M$.
	\end{enumerate}
\end{lemma}

\begin{proof}
	The first assertion is trivial and the second follows easily from items (vi), (vii) and 
	(ix) of Lemma \ref{ll}.
	
To verify assertion (iii), we argue as follows:
According to item (ii), we can suppose that $0\in\icore C,$
where $\icore C$ is the relative interior of $C$ with respect
the topology of $\aff C$ (\cite{Rockafellar1970}).

Let $x,y\in \aff C$ and take $t>0$ such that $tx,\;ty\in C$.
Then, $(tx)\sa(ty),\;(tx)\su(ty)\in C$.

Since $x\sa y=(1/t)((tx)\sa (ty))$ and $x\su y=(1/t)((tx)\su(ty))$, it follows 
that $x\sa y,\;x\su y\in \aff C.$

The proof of the assertion (iv) follows from item (ii) of Lemma \ref{minimal}.
\end{proof}

\begin{corollary}\label{ketdiminv}

Let $x,y\in \R^n$ be incomparable elements. Then, the rectangle $\Pi(x,y)$
with vertices $x$, $y$, $x\sa y$, and $x\su y$ is invariant according to item (ii) of 
Lemma \ref{minimal}. Assume that
$0$ is in the relative interior of $\Pi(x,y)$. Then, the linear hull 
$\Omega(x,y):=\aff \Pi(x,y)$ is an invariant bidimensional
subspace of $\R^m$ by item (iii) of Lemma \ref{latprop}. In this subspace 
$K_0=K\cap\Omega(x,y)$ is a self-dual
lattice cone and $\sa$ and $\su$ restricted to $\Omega(x,y)$ are the 
lattice operations with respect to the order relation that $K_0$ induces in this subspace.
Hence, according to item (iv) of Lemma \ref{latprop} every sublattice in $\Omega(x,y)$
with respect to these lattice operations is an invariant set in $\R^m$.

\end{corollary}

\begin{proof}

We shall use the notation $\cone M$ for the minimal closed convex cone in $\R^m$ containing the
nonempty set $M$.

After a translation in $\Omega(x,y)$, if necessary, we can suppose that 
$x\sa y =0$. Hence, by item (iii) of Lemma \ref{ll} we get $x,y\in K$ and by item
(xii) of the same lemma it follows that $\lng x,y\rng =0.$
We further have that $x,y\in K_0$ and hence $\cone \{x,y\}\subset K_0$. 
In fact we have that $K_0=\cone \{x,y\}.$ Assuming the existence of some 
$z\in K_0\setminus \co \{x,y\},$ it would follow that $\lng x,z\rng <0$,
or $\lng y,z \rng <0.$  In any case we get a contradiction with the
self-duality of $K$. Thus, $K_0$ is a selfdual cone in $\Omega (x,y)$.

In the bidimensional space every generating pointed cone is a latticial cone,
hence so is $K_0$ in $\Omega (x,y)$.
The lattice operations with respect to the order relation $\leq_{K_0}$ induced by $K_0$
in $\Omega (x,y)$
can be characterized geometrically as follows: The infimum $w$ of the set
$\{u,v\}\subset \Omega (x,y)$ is given by the relation $w-K_0=(u-K_0)\cap (v-K_0)$.
By using item (vii) of Lemma \ref{ll}, we can suppose that $u,v\in\Pi(x,y)$.
Therefore, similar ideas to the proof of item (ii) of Lemma \ref{minimal} yield that 
$w=u\sa v$. Analogously, the supremum of the set $\{u,v\}$ with respect to
$\leq_{K_0}$ is exactly $u\su v.$

\end{proof}

\begin{lemma}\label{halfhyper}
The halfspace $H_-$ is invariant 
if and only if the hyperplane $H$ has this property.
\end{lemma}
\begin{proof}
According to item (ii) of Lemma \ref{latprop} we can assume that $0\in H$.

Suppose that $H$ is invariant, but $H_-$ is not. Then, there exist some $x,y\in H_-$
such that $x\su y\notin H_-$ or $x\sa y\notin H_-$. Assume that $x\sa y\notin H_-$.
Then, $x\sa y\in \inte H_+.$ The line segment $[x,x\sa y]$ meets $H$ in
$z=\lambda x +(1-\lambda)x\sa y,\;\lambda \in ]0,1],$ the line segment
$[y,x\sa y]$ meets $H$ at $w=\mu y  +(1-\mu)x\sa y,\;\mu \in ]0,1].$
According to item (xi) in Lemma \ref{ll} we have then
\[z\sa w=x\sa y\notin H,\]
which contradicts the invariance of $H$.

Suppose now that $H_-$ is invariant, but $H$ is not. Then, there exist some $x,y\in H$
such that $x\su y\notin H$ or $x\sa y\not \in H$. Since $H_-$ is invariant, we can
assume that $x\su y\in \inte H_-$. Let $u$ be the normal of $H$. Then, $\lng u,x\su y\rng<0.$
By using the relation in item (v), we have then
\[0=\lng u,x+y\rng=\lng u, x\su y\rng +\lng u,x\sa y\rng.\]
Whereby, by using the relation $\lng u,x\su y\rng <0,$
we conclude that
\[\lng u,x\sa y\rng >0,\]
that is, $x\sa y\in \inte H_+$, contradicting the invariance of $H_-$.
\end{proof}

\begin{lemma}\label{faceof}
	If the nonempty closed convex set $C\subset\R^m$ is invariant, 
	then so is every face $C_0$ of $C$.
\end{lemma}

\begin{proof}
	Take $x,y\in C_0$. Then, $(1/2)(x+y)\in C_0$,
	because $C_0$ is convex. Using the standard relation 
	\[\frac{1}{2}(x+y)= \frac{1}{2}x\sa y+\frac{1}{2}x\su y,\]
	the inclusions $x\sa y,\;x\su y\in C$,
	and the definition of the face, we have that
	\[x\sa y,\;x\su y \in C_0.\]
\end{proof}

\begin{lemma}\label{ls}
	A linear subspace $S$ of $\R^m$ is invariant 
	if and only if $S$ is invariant with respect to $P_K$, i.e., 
	$P_K(S)\subset S$.
\end{lemma}

\begin{proof}
	Suppose that $S$ is invariant and let any 
	$x\in S$. Then, from $0\in S$ and item (i) of Lemma \ref{ll}, it follows that 
	$P_K(x)=0+P_K(x-0)=0\su x\in S$.  

	Conversely, suppose that $P_K(S)\subset S$. Hence, by using again item (i) of 
	Lemma \ref{ll} and the invariance of a linear subspace under linear combinations,
	for any $x,y\in S$ we have $x\sa y=x-P_K(x-y)\in S$ and $x\su y=x+P_K(y-x)\in S$. 
\end{proof}

Denote by $\bdr C$ the boundary of a set $C$.

\begin{lemma}\label{tanginvar}
Suppose that $C$ is an invariant closed convex set with nonempty interior,
and $H$ is a hyperplane tangent to $C$ in some point of $\bdr C$.
Then, $H$ is invariant.
\end{lemma}
\begin{proof}
According to item (ii) of Lemma \ref{latprop} we can assume that $0\in \bdr C$, that
$H$ is tangent to $C$ at $0$, and that $C\subset H_-$.

We shall prove our claim by contradiction: we assume that $H$ is not
invariant.

Since $H$ is not invariant, there exist some $z,\;w\in H$ such that
$z\su w$ or $z\sa w$ is not in $H$. Suppose that $u$ is the normal
of $H$. From the relation in item (v) of Lemma \ref{ll} we have then
\[0=\lng u,z+w\rng= \lng u,z\su w \rng+\lng u,z\sa w \rng,\]
whereby it follows that $z\su w$ and $z\sa w$ are in opposite open
half-spaces determined by $H$.

Suppose that $z\sa w\in \inte H_+.$ Taking $x=z-(z+w)/2,$
we have $-x=w-(z+w)/2.$ Then, by our working hypothesis that $0\in H,$
it follows that the line segment $[-x,x]\subset H.$ We can easily check 
that $(-x)\sa x\in \inte H_+.$ Denoting by $B$ the unit ball in $\R^m$,
then there exists some $\delta >0$ such that
\begin{equation}\label{gomb}
(-x)\sa x+\delta B\subset \inte H_+.
\end{equation}

We have the relation
\[[-x,x]=\{tx:\;t\in[-1,1]\}.\]
Next we project $[-x,x]$ in the direction of $u$ onto $\bdr C$.
All the above reasonings are valid when we change $x$
with its positive multiple, hence we can chose $x$ small enough,
so that the above projection to make a sense.

Denote by $\gamma (t)$ the image of $tx$ in $\bdr C$ by  this projection.
Since $H$ is a tangent hyperplane, the segment $[-x,x]$ will be tangent to
$\gamma $ at $t=0,\;\gamma (0)=0,$ $\gamma'(0)$ exists, and $\gamma'(0)=x.$

Since $\gamma$ is differentiable in $t=0$, we have the following
representations around $0$:
\begin{equation}\label{post}
\gamma (t)=tx+\eta (t),\;t>0,
\end{equation}
and
\begin{equation}\label{negt}
\gamma (-t)=-tx+\zeta (-t),\;t>0,
\end{equation}
where
\begin{equation}\label{kiso}
\frac{\eta (t)}{t}\to 0\;\;\textrm{and}\;\; \frac{\zeta (-t)}{t}\to 0,\;\;\textrm{as}\;t\to 0,\;t>0.
\end{equation}

Using item (x) of Lemma \ref{ll}, as well as the relations (\ref{post}) and (\ref{negt}), we have then
\[\|(-tx)\sa (tx)- \gamma (-t)\sa \gamma (t)\|\leq \frac{3}{2}(\|-tx-\gamma(-t)\|+\|tx-\gamma (t)\|)=
\frac{3}{2}(\|\zeta (-t)\|+\|\eta (t)\|).\]

Dividing the last relation by $t>0$,  and using the relation in item (vii) of
Lemma \ref{ll}, we obtain that
\begin{equation}\label{becsles}
\|(-x)\sa x-\frac{1}{t} \gamma(-t)\sa \gamma (t)\|\leq \frac{3}{2}\lf(\lf\|
\frac{\zeta(-t)}{t}\rg\|+\lf\|\frac{\eta (t)}{t}\rg\|\rg).
\end{equation}

Take now $t>0$ small enough in order to have by (\ref{kiso})
\[\frac{3}{2}\lf(\lf\|\frac{\zeta(-t)}{t}\rg\|+\lf\|\frac{\eta (t)}{t}\rg\|\rg)<\delta.\]
For such a $t>0$ we have, by using (\ref{becsles}), that
\[\frac{1}{t}(\gamma (-t)\sa \gamma (t))\in \inte H_+,\]
and thus
\[\gamma (-t)\sa \gamma (t) \in \inte H_+,\]
that is, $\gamma (-t),\;\gamma (t)\in C$, but
\[\gamma (-t)\sa \gamma (t)\notin C,\]
contradicting the invariance of $C$.

The obtained contradiction shows that $H$ must be invariant
with respect to the operations $\su$ and $\sa$.

\end{proof}

\section{Isotonicity of the projection onto a hyperplane}\label{sec iso}

\begin{lemma}\label{ti}
	Let $H\subset\R^m$ be a hyperplane 
	through the origin with unit normal vector $u\in\R^m$. Then, $P_H$ is isotone
	if and only if \[\lng x,y\rng\ge\lng u,x\rng\lng u,y\rng,\] for any $x,y\in K$.
\end{lemma}
\begin{proof}
	Since $P_H$ is linear, it follows that $P_H$ is isotone if and only if 
	\begin{equation}\label{eih}
		P_Hx=x-\lng u,x\rng u\in K,
	\end{equation} 
	for any $x\in K$. By the self-duality of $K$, it follows that relation 
	\eqref{eih} is equivalent to 
	\[
	\lng x,y\rng=\lng u,x\rng\lng u,y\rng+\lng x-\lng u,x\rng u,y\rng
	\ge\lng u,x\rng\lng u,y\rng,
	\] 
	for any $x,y\in K$.
\end{proof}
\begin{lemma}\label{tii}
	Let $H\subset\R^m$ be a hyperplane 
	through the origin with unit normal vector $u\in\R^m$. If $P_H$ is isotone, 
	then $H$ is invariant.
\end{lemma}
\begin{proof}
	By Lemma \ref{ls} it is enough to show that if 
	$\lng u,z\rng=0$, then $\lng u,P_Kz\rng=0$. Suppose that $\lng u,z\rng=0$.
	Then, $P_Kz\in K$, $P_Kz-z=P_K(-z)\in K$ and 
	$\lng P_Kz-z,P_Kz\rng=0$ by Lemma \ref{lm}. By using
	Lemma \ref{ti}, with $x=P_Kz$ and $y=P_Kz-z$, we get
	\[
	\lng u, P_Kz\rng^2=\lng u,P_Kz\rng\lng u,P_Kz-z\rng
	\le\lng P_Kz,P_Kz-z\rng=0.
	\]
	Hence, it follows that $\lng u,P_Kz\rng=0$.
\end{proof}

\vspace{3mm}

\section{The proof of Theorem \ref{FOO}}\label{spt1}

It is known (see e.g. \cite{Rockafellar1970}, Theorem 25.5) that if $C\subset \R^m$ is a closed convex set
with nonempty interior, then $\bdr C$ contains a dense subset
of points where this surface is differentiable. Since the topology
of $\bdr C$ possesses a countable basis, we can select from this dense set a countable
dense set $\{a_i:\;i\in \N\}\subset \bdr C$ such that
there exist the tangent hyperplanes $H(u_i,a_i)$ to $C$ and
$C\subset H_-(u_i,a_i),\; i\in \N.$ Since the set $\{a_i,\;i\in \N\}$ is dense
in $\bdr C$, a standard convex geometric reasoning shows that in fact
\begin{equation}\label{vegsob}
C=\cap_{i\in \N} H_-(u_i,a_i).
\end{equation}
  
Now, if $C$ is invariant, then so is $H(u_i,a_i),\;i\in \N$	
by Lemma \ref{tanginvar}.
Hence, the necessity of the condition in Theorem \ref{FOO} is proved.

Conversely, if we have the representation (\ref{vegsob}) with the
hyperplanes $H(u_i,a_i),\;i\in \N$ invariant, 
 then, by Lemma \ref{halfhyper}, the halfspaces $H_-(u_i,a_i),\;i\in \N$ are also invariant.
Then, by using item (i) of Lemma \ref{latprop} and the representation (\ref{vegsob}), we see 
that $C$ is invariant with respect to the operations $\sa$ and $\su$ and the sufficiency of
Theorem \ref{FOO} is proved.

\section{The proof of Theorem \ref{ISO}}\label{spt2} 	
	
	Assume that the closed convex set $C$ is invariant .
	Let $x,y\in \R^m$ with $x\le y$ and denote $u=P_Cx$, $v=P_Cy$.
	
	Assume that $u\le v$ is false. Then, from $u\su v\in C$, the definition of the
	projection and item (iii) of Lemma \ref{ll}, we have $\|y-v\|<\|y-u\su v\|$. Hence,
	from
	\[\|y-v\|^2=\|y-u\su v\|^2+\|u\su v-v\|^2+2\lng y-u\su v,u\su v-v\rng,\] 
	it follows that 
	\[\|u\su v-v\|^2<2\lng u\su v-y,u\su v-v\rng.\]
	On the other hand, since $u\sa v\in C$, we have $\|x-u\|\le\|x-u\sa v\|$, and thus we 
	have similarly that
	\[\|u\sa v-u\|^2\le2\lng u\sa v-x,u\sa v-u\rng.\]
	Summing up the latter two inequalities and using item (v) of Lemma \ref{ll}), it 
	follows that
	\[\lng u\su v-v,u\su v-v\rng=\|u\su v-v\|^2<\lng u\su v-y,u\su v-v\rng+
	\lng x-u\sa v,u\su v-v\rng.\]
	Thus, 
	\[\lng y-x-(v-u\sa v),u\su v-v\rng<0.\]
	Combining the latter inequality with item (viii) of Lemma \ref{ll}, we obtain that 
	\[\lng y-x,u\su v-v\rng<0.\] But this is a contradiction, because $y-x\in K=K^*$ and
	$u\su v-v\in K$ (by item (iii) of Lemma \ref{ll}).
	
	The obtained contradiction shows that $P_C$ must be isotone.

\begin{corollary}\label{hypereqv}
Let $H$ be a hyperplane in $\R^m$. Then, $H$ is invariant if and only if it is an isotone projection set.
\end{corollary}
\begin{proof}
The proof follows from the joint application of Theorem \ref{ISO} and Lemma \ref{tii}.
\end{proof}

\section{The proof of Theorem \ref{POLYH}}\label{spt3}

Let us verify first the following equivalences
\begin{equation}\label{paratlan}
\textrm{(i)} \Leftrightarrow \textrm{(iii)} \Leftrightarrow \textrm{(v)}.
\end{equation}

From Theorem \ref{FOO} we have the equivalence 
\[\textrm{(i)} \Leftrightarrow \textrm{(iii)}.\]

From Lemma \ref{faceof} it follows

\[\textrm{(i)}\Rightarrow \textrm{(v)}.\]

If (v) holds then every hyperface $C\cap H(u_i,a_i)$ must be invariant.
But then,  as $H(u_i,a_i)$ is the affine hull of this hyperface, it
must be invariant too, by item (iii) of Lemma \ref{latprop}. Hence
\[\textrm{(v)}\Rightarrow \textrm{(iii)}\]
and (\ref{paratlan}) has been verified.

From Theorem \ref{ISO} we have
\[\textrm{(i)}\Rightarrow \textrm{(ii)}.\]

We shall show next, that 
\[\textrm{(ii)}\Rightarrow \textrm{(iv)}.\]
Assume the contrary: $C$ is an isotone projection set, but some hyperplane
$H=H(u_i,a_i)$ in its sharp representation is not.

Bearing in mind item (ii) of Lemma \ref{latprop}, we can assume that $0$ is in the relative 
interior of the hyperface
$F=C\cap H$. If $B$ denotes the unit ball in $\R^m$, then for an appropriate 
positive $\delta >0$ we can realize that
\[H\cap \delta B\subset F.\]

Since $C\subset H_-$, for each element $z\in \delta B\cap H_+$ we have 
\begin{equation}\label{lapproj}
P_Cz=P_Hz\in F.
\end{equation}
Indeed, from $P_Hz=P_{H_-}z$ and $C\subset H_-$ we have on the one hand
\begin{equation}\label{onehand}
\|z-P_Hz\|=\|z-P_{H_-}z\|\leq \|z-P_Cz\|,
\end{equation}
and on the other hand $P_Hz\in \delta B\cap H\subset F\subset C$
(as $P_H$ is nonexpansive) and then
\begin{equation}\label{otherhand}
\|z-P_Hz\|\geq \|z-P_Cz\|.
\end{equation}
The relations (\ref{onehand}) and (\ref{otherhand}) yield
\[\|z-P_Cz\|=\|z-P_Hz\|,\]
which together with $P_Hz\in C$ and the unicity of the best
approximation conclude that $P_Cz=P_Hz.$

From our working hypothesis that $P_H$ is not isotone and the
linearity of this mapping (from the condition $0\in H$), this is equivalent with the existence of some
$z\in K$ with $P_Hz\notin K.$  The same is true for any positive multiple of $z$. 
Hence, we can assume at once that $z\in \delta B.$

Suppose that $z\in H_+$. From the isotonicity of $P_C$, we have as $0\leq z$ and $P_C(0)=0,$ 
that \[P_Cz=P_Cz-P_C(0)\in K,\]
which is impossible since by (\ref{lapproj})
\[P_Cz=P_Hz\notin K.\]

Suppose that $z\in H_-$. Then, $-z\in \delta B\cap H_+$ and then
\begin{equation}\label{minusz}
P_C(-z)=P_H(-z)=-P_Hz\notin -K.
\end{equation}

Since $-z\leq 0$, the isotonicity of $P_C$ yields
\[P_C(-z) \leq P_C(0)=0\] and hence
$P_C(-z)\in -K$, contradicting (\ref{minusz}).

The obtained contradictions conclude that $P_H$ must be isotone.

The relation 

\[\textrm{(iv)}\Rightarrow \textrm{(iii)}\]
is a direct consequence of Lemma \ref{tii}.

\section{Particular case: the Lorentz cone}\label{sect lo}

\begin{lemma}\label{ltr}
	For any $x,y,a\in\R^m$ the following inequality holds:
	\begin{equation}\label{etr}
		(\lng x,y\rng+\|x\|\|y\|)\|a\|^2\ge\lng a,x\rng\lng a,y\rng
	\end{equation}
\end{lemma}

\begin{proof}
	Denote by $\varphi,\theta,\rho\in[0,\pi]$ the angles of the vectors
	$\{a,x\}$, $\{a,y\}$ and $\{x,y\}$, respectively in radians. Then, it is known that 
	$\rho\le\varphi+\theta$. Since the cosine function is decreasing in the interval 
	$[0,\pi]$, the latter inequality gives
	\[
	\cos\rho\ge\cos(\varphi+\theta)=\cos\varphi\cos\theta-\sin\varphi\sin\theta\ge
	\cos\varphi\cos\theta-1.
	\] 
	Thus, $\cos\rho+1\ge\cos\varphi\cos\theta$, from where it follows 
	\[\f{\lng x,y\rng}{\|x\|\|y\|}+1\ge\f{\lng a,x\rng}{\|a\|\|x\|}\f{\lng a,y\rng}
	{\|a\|\|y\|},\] or equivalently inequality \eqref{etr}.
\end{proof}

\begin{lemma}\label{tli}
	Let $m>1$ and $K\subset\R^{m+1}$ be the Lorentz cone
	\[K=\{(x,x^{m+1})\in \R^m\otimes\R:\;\|x\|\leq x^{m+1}\},\]
  and $H\subset\R^{m+1}$ a hyperplane 
	through the origin with unit normal vector $(a,a^{m+1})$, where $a\in\R^m$ and
	$a^{m+1}\in\R$. Then, $P_H$ is isotone if and only if $a^{m+1}=0$.
\end{lemma}
\begin{proof}
	Let $b=(a,a^{m+1})$. By Lemma \ref{ti}, we have to show that for any 
	$u=(x,x^{m+1})\in K$ and $v=(y,y^{m+1})\in K$ we have 
	$\lng u,v\rng-\lng b,u\rng\lng b,v\rng\ge0$ if and only if $a^{m+1}=0$.
	Suppose that $a^{m+1}=0$. Then, by using Lemma \ref{ltr}, we have 
	\begin{eqnarray*}
		\lng u,v\rng-\lng b,u\rng\lng b,v\rng
		=\lng x,y\rng+x^{m+1}y^{m+1}-\lng a,x\rng\lng a,y\rng\\
		\ge(\lng x,y\rng+\|x\|\|y\|)\|a\|^2-\lng a,x\rng\lng a,y\rng\ge0
	\end{eqnarray*}
	Conversely, suppose that for any $u,v\in K$
	we have $\lng u,v\rng-\lng b,u\rng\lng b,v\rng\ge0$. Since $m>1$, 
	there exists $z\in\R^m$ such that $\lng a,z\rng=0$ and $\|z\|=1$. Let $u=(z,1)$ 
	and $v=(-z,1)$. Then, $u,v\in K$ and
	thus \[0\le\lng u,v\rng-\lng b,u\rng\lng b,v\rng=
	-\|z\|^2+1-(\lng a,z\rng+a^{m+1})(-\lng a,z\rng+a^{m+1})=-(a^{m+1})^2.\]
	Therefore, $a^{m+1}=0$.
\end{proof}

	Bearing in mind, item (ii) of Lemma \ref{latprop}, the working hypotheses $0\in H$ and 
	$\|(a,a^{m+1})\|=1$ can be ignored in the applications of the above lemma.

\begin{corollary}
Let $M$ be a closed convex subset with nonempty interior in $\R^{m+1}=\R^m\otimes \R$ with $m>1$.
Consider the following assertions:
\begin{enumerate}
\item[\emph{(i)}]  $M$ is invariant
with respect to the operations $\sa$ and $\su$ defined by the Lorentz
cone $K$,
\item[\emph{(ii)}] $M$ is an isotone projection set,
\item[\emph{(iii)}]
\begin{equation}\label{ahenger} 
M=C\times \R, 
\end{equation}
where $C$ is a closed convex set with nonempty interior in $\R^m$.
\end{enumerate}
Then
\[\emph{(iii)}\Leftrightarrow \emph{(i)}\Rightarrow \emph{(ii)}.\]
\end{corollary}

\begin{proof}
From the convex geometry it follows that if $M$ is of the form (\ref{ahenger}),
then it can be represented as
\begin{equation}\label{hyperhenger}
M=\cap_{i\in \N} H_-((a_i,0),(b_i,b_i^{m+1})).
\end{equation}
Since every hyperplane $H((a_i,0),(b_i,b_i^{m+1}))$ is isotone by Lemma \ref{tli},
it follows from Corollary \ref{hypereqv} that each $H((a_i,0),(b_i,b_i^{m+1}))$ is invariant too. 
But then according to Theorem \ref{FOO},
$M$ is an invariant set. The usage of Theorem \ref{ISO} then shows that $M$ is an isotone projection set.

If $M$ is invariant, by Theorem \ref{FOO} and Lemma \ref{tli} it must be of form (\ref{hyperhenger}).
Putting
\[C=\R^m\cap (\cap_{i\in N} H_-((a,0),(b,b^{m+1})),\]
we arrive to the required representation (\ref{ahenger}) of $M$.

\end{proof}

\begin{remark}
	$\,$

	\begin{enumerate}
		\item The implication \emph{(iii)}$\Rightarrow$\emph{(ii)} of the above corollary 
			can be shown directly as well. Indeed, by the definition of the projection
			it easily follows that \[P_M(x,x^{m+1})=P_{C\times\R}(x,x^{m+1})=
			(P_Cx,x^{m+1}),\] because for any $(y,y^{m+1})\in M=C\times\R$ 
			we have 
			\begin{eqnarray*}
				\begin{array}{rcl}
					\|(y,y^{m+1})-(x,x^{m+1})\|^2
					&=&\|y-x\|^2+|y^{m+1}-x^{m+1}|^2\ge
					\|P_Cx-x\|^2\\
					&=&\|(P_Cx,x^{m+1})-(x,x^{m+1}\|^2
				\end{array}
			\end{eqnarray*} 
			and $(P_Cx,x^{m+1})\in C\times\R$. Now let $(x,x^{m+1})\le (y,y^{m+1})$.
			Then, \[\|y-x\|\le y^{m+1}-x^{m+1}.\] On the other hand, by the 
			nonexpansivity
			of the projection $P_C$, we have \[\|P_Cy-P_Cx\|\le\|y-x\|.\] Thus, the
			latter two inequalities imply $\|P_Cy-P_Cx\|\le y^{m+1}-x^{m+1}$, or
			equivalently \[P_M(x,x^{m+1})=(P_Cx,x^{m+1})\le (P_Cy,y^{m+1})
			=P_M(y,y^{m+1}).\] Hence, $P_M$ is isotone. 
		\item In the case $m=1$ the Lorentz cone $K$ is nothing else as the rotated 
			$\R^2_+$ and hence in this case the investigations of the next section 
			take effect .
		\item The conditions $m>1$ and that the interior of the convex set is nonempty 
			is essential in the assertions of the corollary above. By Corollary 
			\ref{ketdiminv} and 
			the next section, it can be seen that the invariant sets of dimension 2 
			can have a different shape.
	\end{enumerate}
\end{remark}

\section{Particular case: the cone $\R^m_+$}\label{sect no}

In this case the invariant  sets are the sublattices of the coordinate-wise ordered
Euclidean space. The following lemma is the sufficiency part of Lemma 2.1 in
\cite{NishimuraOk2012}. We include here its proof for the sake of completeness.

\begin{lemma}\label{NishOk}
If the closed convex set $C\subset \R^m$ admits an isotone projection $P_C$ with
respect to the coordinate-wise order in $\R^m$, then $C$
is a sublattice.
\end{lemma}

\begin{proof}
Suppose that $P_C$ is isotone and take $x,y\in C$.
Let us see that $x\vee y\in C.$ 

From the characterization (\ref{charac}) of the projection we have
\begin{equation}\label{NishOkf}
\lng P_C(x\vee y)-x\vee y,P_C(x\vee y)-y\rng \leq 0.
\end{equation}
Since $x\leq x\vee y$ and $P_C$ is isotone, it follows that $x=P_Cx\leq P_C(x\vee y)$. Similarly,
$y\leq P_C(x\vee y)$ and hence $x\vee y\leq P_C(x\vee y).$ We have also
\begin{equation}\label{NO}
0\leq P_C(x\vee y)-x\vee y\leq P_C(x\vee y)-y. 
\end{equation}
The two terms in the scalar product (\ref{NishOkf}) are in $K=\R^m_+,$
and since $K$ is self-dual, we must have the equality:
\begin{equation}\label{NishOkff}
\lng P_C(x\vee y)-x\vee y,P_C(x\vee y)-y\rng =0.
\end{equation}

By using again the self-duality of $K$, the relation (\ref{NO}), as well as (\ref{NishOkff}), 
it follows that
\[0\leq \lng P_C(x\vee y)-x\vee y, (P_C(x\vee y)-y)- (P_C(x\vee y)-x\vee y)\rng=
-\|P_C(x\vee y)-x\vee y\|^2,\]
thus we must have
\[P_C(x\vee y)=x\vee y,\]
and since $C$ is closed, $x\vee y\in C.$

Similar reasonings show that $x\wedge y\in C.$
\end{proof}

\begin{lemma}\label{foo}
The hyperplane $H$ through $0$ with the normal $u=(u^1,...,u^m)$
is a sublattice if and only if
\[u^iu^j \leq 0,\;\;\textrm{whenever}\;i\not=j.\]
\end{lemma}
\begin{proof}
By Corollary \ref{hypereqv} it is enough to prove that $P_H$ is isotone if and only if
the conditions of the lemma hold.

In the following reasoning, for sake of simplicity, suppose that 
$\|u\|=1.$
Since $P_H$ is linear, in order to characterize the hyperplane
$H$ with the property that $x\leq y$ implies $P_Hx\leq P_Hy$, it is
sufficient to give necessary and sufficient conditions on
the unit vector $u$ such that
\begin{equation}\label{charH}
P_H e_i\geq 0,\;\;i=1,...,m,
\end{equation}
where $e_i=(0,...,0,1,0...0), \;i=1,...,m$ are the standard unit vectors
of the Cartesian reference system.

Since $u$ is a unit vector, the conditions (\ref{charH})
can be written in the form:
\begin{equation}\label{charu}
P_He_i =e_i-\lng u,e_i\rng u=(0,...,0,1,0,...,0)-u^i(u^1,...,u^m)\geq 0,\;i=1,...,m.
\end{equation}
These conditions yield
\begin{equation}\label{nonpositive}
u^iu^j \leq 0,\;\textrm{whenever}\; i\not=j,
\end{equation}
and
\begin{equation}\label{triv}
1-(u^i)^2\geq 0,\;i=1,...,m.
\end{equation}
But the conditions (\ref{triv}) are trivially satisfied
by the condition $\|u\|=1$.

If $\|u\|\not= 1,$ we can carry out the proof with $u/\|u\|$
in place of $u$ and we get the same conditions (\ref{nonpositive})
on the coordinates of $u$.

\end{proof}

By putting together Theorem \ref{FOO}, Lemma \ref{NishOk} and Lemma \ref{foo}, we obtain the 
following corollary:

\begin{corollary}

Let $C$ be a closed convex set with nonempty interior of the coordinate-wise ordered
Euclidean space $\R^m$. Then, the following assertions are equivalent
\begin{enumerate}
\item [\emph{(i)}] The set $C$ is a sublattice;
\item [\emph{(ii)}]The projection $P_C$ is isotone;
\item [\emph{(iii)}]\begin{equation*}
C=\cap_{i=\N} H_-(u_i,a_i),
\end{equation*}
where each hyperplane $H(u_i,a_i)$ is tangent to $C$ and the normals $u_i$ 
are nonzero vectors $u_i=(u_i^1,...,u_i^m)$ with the properties 
	$u_i^ku_i^l\leq 0$
	whenever $k\not= l,\;\;i\in \N.$

\end{enumerate}
\end{corollary}

The equivalence of items (i) and (iii) says slightly more than the main result
in \cite{QueyranneTardella2006}.
 
\section{Comments and open questions}\label{concl}
Motivated by isotone iterative methods for variational inequalities, the second author
put the following very general and still open question: Which are the closed convex sets
which possess a projection onto them which is isotone with respect to an order relation
defined by a given cone? A related at least as interesting question is: Which are the 
closed convex sets for which there exist a cone such that the projection onto them are 
isotone with respect to order relation defined by the cone? Although these very general 
questions seem extremely 
difficult to handle, the present paper partially answered the first question for 
self-dual cones. The investigation led to interesting connections with the invariant sets 
with respect to the extended lattice operations defined by a self-dual cone. Another 
question is: Can this invariance approach be extended for more general cones, e.g., by 
introducing extended lattice operations with respect to both the cone and its dual? 
We expect this paper to open a new area, providing a general tool for studying 
variational inequalities and related equilibrium problems by using isotonicity with 
respect to orders defined by cones, and greatly widening the field of similar previous 
investigations.

\bibliographystyle{abbrv}
\bibliography{selfdu}

\begin{thebibliography}{10}

\bibitem{Auslender1976}
A.~Auslander.
\newblock {\em Optimization M\'ethodes Num\'eriques}.
\newblock Masson, Paris, 1976.

\bibitem{BarkerForan1976}
G.~P. Barker and J.~Foran.
\newblock Self-dual cones in {E}uclidean spaces.
\newblock {\em Linear Algebra Appl.}, 13:147--155, 1976.

\bibitem{Bertsekas1989}
D.~P. Bertsekas and J.~N. Tsitsiklis.
\newblock {\em Parallel and Distributed Computation: Numerical Methods}.
\newblock Prentice-Hall, Inc, Englewood Cliffs, New Jersey, 1989.

\bibitem{GowdaSznajderTao2004}
M.~S. Gowda, R.~Sznajder, and J.~Tao.
\newblock Some p-properties for linear transformations on {Euclidean} {Jordan}
  algebras.
\newblock {\em Linear Algebra Appl.}, 393:203--232, 2004.

\bibitem{Isac1996}
G.~Isac.
\newblock On the order monotonicity of the metric projection operator.
\newblock {\em Approximation Theory, Wavelets and Applications, ed. S. P.
  Singh}, 1995.

\bibitem{IsacNemeth1986}
G.~Isac and A.~B. N\'emeth.
\newblock Monotonicity of metric projections onto positive cones of ordered
  {E}uclidean spaces.
\newblock {\em Arch. Math.}, 46(6):568--576, 1986.

\bibitem{IsacNemeth1990b}
G.~Isac and A.~B. N\'emeth.
\newblock Isotone projection cones in {H}ilbert spaces and the complementarity
  problem.
\newblock {\em Boll. Un. Mat. Ital. B.}, 7(4):773--802, 1990.

\bibitem{IsacNemeth2008c}
G.~Isac and S.~Z. N\'emeth.
\newblock Regular exceptional family of elements with respect to isotone
  projection cones in {Hilbert} spaces and complementarity problems.
\newblock {\em Optim Lett.}, 2(3):567--576, 2008.

\bibitem{Iusem1997}
A.~N. Iusem and B.~F. Svaiter.
\newblock A variant of {Korpelevich's} method for variational inequalities with
  a new search strategy.
\newblock {\em Optimization}, 42(4):309--321, 1997.

\bibitem{Khobotov1987}
E.~N. Khobotov.
\newblock A modification of the extragradient method for solving variational
  inequalities and some optimization problems.
\newblock {\em Zhurnal Vychislitel'noi Matematiki i Matematicheskoi Fiziki},
  27(10):1462--1473, 1987.

\bibitem{Korpelevich1976}
G.~M. Korpelevich.
\newblock The extragradient method for finding saddle points and other
  problems.
\newblock {\em Matecon}, 12:747--756, 1976.

\bibitem{Marcotte1991}
P.~Marcotte.
\newblock Application of {Khobotov's} algorithm to variational inequalities and
  network equilibrium problems.
\newblock {\em Information Systems and Operational Research}, 29:258--270,
  1991.

\bibitem{Moreau1962}
J.~J. Moreau.
\newblock D\'ecomposition orthogonale d'un espace hilbertien selon deux c\^ones
  mutuellement polaires.
\newblock {\em C. R. Acad. Sci.}, 255:238--240, 1962.

\bibitem{Nagurney1993}
A.~Nagurney.
\newblock {\em Network Economics - A Variational Inequality Approach}.
\newblock Kluwer Academic Publishers, Dordrecht, The Netherlands, 1993.

\bibitem{NemethNemeth2009}
A.~B. N\'emeth and S.~Z. N\'emeth.
\newblock How to project onto an isotone projection cone.
\newblock {\em Linear Algebra Appl.}, 433(1):41--51, 2010.

\bibitem{Nemeth2009a}
S.~Z. N\'emeth.
\newblock Iterative methods for nonlinear complementarity problems on isotone
  projection cones.
\newblock {\em J. Math. Anal. Appl.}, 350(1):340--347, 2009.

\bibitem{NishimuraOk2012}
H.~Nishimura and E.~A. Ok.
\newblock Solvability of variational inequalities on {Hilbert} lattices.
\newblock {\em Preprint}, pages 1--28, 2012.

\bibitem{QueyranneTardella2006}
M.~Queyranne and F.~Tardella.
\newblock Bimonotone linear inequalities and sublattices of $\mathbb{R}^n$.
\newblock {\em Linear Algebra Appl.}, 413:100--120, 2006.

\bibitem{Rockafellar1970}
R.~T. Rockafellar.
\newblock {\em Convex Analysis}.
\newblock Princeton: Princeton Univ. Press, 1970.

\bibitem{Sibony1970}
M.~Sibony.
\newblock M\'ethodes it\'eratives pour les \'equations et in\'equations aux
  d\'erivées partielles non lin\'eaires de type monotone.
\newblock {\em Calcolo}, 7:65--183, 1970.

\bibitem{Solodov1999}
M.~V. Solodov and B.~F. Svaiter.
\newblock A new projection method for variational inequality problems.
\newblock {\em SIAM J. Control Optim}, 37(3):765--776, 1999.

\bibitem{Solodov1996}
M.~V. Solodov and P.~Tseng.
\newblock Modified projection-type methods for monotone variational
  inequalities.
\newblock {\em SIAM J. Control Optim}, 34(5):1814--1830, 1996.

\bibitem{Sun1996}
D.~Sun.
\newblock A class of iterative methods for nonlinear projection equations.
\newblock {\em J. Optim. Theory Appl.}, 91(1):123--140, 1996.

\bibitem{Topkis1976}
D.~M. Topkis.
\newblock The structure of sublattices of the product of n lattices.
\newblock {\em Pacific J. Math.}, 65:525--532, 1976.

\bibitem{Veinott1981}
A.~F. Veinott.
\newblock Reprezentation of general and polyhedral sublattices and sublattices
  of product spaces.
\newblock {\em Linear Algebra Appl.}, 114/115:172--178, 1981.

\bibitem{Zarantonello1971}
E.~Zarantonello.
\newblock Projections on convex sets in {Hilbert} space and spectral theory,
  {I}: {P}rojections on convex sets, {II}: {S}pectral theory.
\newblock {\em Contrib. Nonlin. Functional Analysis, Proc. Sympos. Univ.
  Wisconsin, Madison}, pages 237--424, 1971.

\end{thebibliography}

\end{document}